\documentclass{article}
\usepackage{geometry} %[margin=1in]
\usepackage{amsmath,amssymb,bm,cite}
\usepackage[dvipdfmx]{hyperref}

\makeatletter
\newenvironment{proof}[1][\proofname]{\par\normalfont
  \topsep6pt plus6pt\trivlist\item[\hskip\labelsep\itshape
  #1\@addpunct{:}]\ignorespaces}{\qed\endtrivlist}
\newcommand{\proofname}{Proof}
\DeclareRobustCommand{\qed}{%
  \ifmmode% if math mode, assume display: omit penalty etc.
  \else\leavevmode\unskip\penalty9999\hbox{}\nobreak\hfill\fi
  \quad\hbox{\qedsymbol}}
\newcommand{\qedsymbol}{\openbox}
\newcommand{\openbox}{\leavevmode\hbox to.77778em{%
    \hfil\vrule\vbox to.675em{%
      \hrule width.6em\vfil\hrule}\vrule\hfil}}
\makeatother

\newcommand{\K}{\mathcal{K}}
\newcommand{\Exp}{\mathbb{E}}
\newcommand{\Cov}{\mathsf{Cov}}
\newcommand{\C}{\mathbb{C}}\newcommand{\N}{\mathbb{N}}
\newcommand{\R}{\mathbb{R}}
\newcommand{\dd}{\mathrm{d}}
\newcommand{\Lpl}{\mathcal{L}}
\newcommand{\Tr}{\mathsf{Tr}}\newcommand{\Det}{\mathsf{Det}}
\newcommand{\op}{\mathsf{op}}
\newcommand{\ind}[1]{\bm{1}_{#1}}
\let\Bar\overline
\let\Tilde\widetilde

\newtheorem{proposition}{Proposition}
\newtheorem{theorem}{Theorem}
\newtheorem{lemma}{Lemma}
\newtheorem{remark}{Remark}

\title{Scaling limits of stationary determinantal shot-noise
  fields}
\author{Takumi Aburayama and Naoto Miyoshi\thanks{%
  Corresponding author: Department of Mathematical and Computing
  Science, Tokyo Institute of Technology, 2-12-1-W8-52 Ookayama,
  Tokyo 152-8552, Japan. E-mail: miyoshi@is.titech.ac.jp}
  \thanks{%
    The second author's work was supported by the Japan Society
    for the Promotion of Science Grant-in-Aid for Scientific
    Research~(C) 19K11838.}
  \\ Tokyo Institute of Technology}
\begin{document}\sloppy\allowbreak\allowdisplaybreaks
\maketitle

\begin{abstract}
We consider a shot-noise field defined on a stationary determinantal
point process on $\R^d$ associated with i.i.d.\ amplitudes and a
bounded response function, for which we investigate the scaling limits
as the intensity of the point process goes to infinity.
Specifically, we show that the centralized and suitably scaled
shot-noise field converges in finite dimensional distributions to i)~a
Gaussian random field when the amplitudes have the finite second
moment and ii)~an $\alpha$-stable random field when the amplitudes
follow a regularly varying distribution with index~$-\alpha$ for
$\alpha\in(1,2)$.
We first prove the corresponding results for the shot-noise field
defined on a homogeneous Poisson point process and then extend them to
the one defined on a stationary determinantal point process.
\end{abstract}

\section{Introduction}\label{sec:Intro}

Let $\Phi_\lambda = \sum_{n=1}^\infty\delta_{X_n}$ denote a simple and
stationary point process on $\R^d$, $d\in\N=\{1,2,\ldots\}$, with
intensity~$\lambda = \Exp[\Phi_\lambda([0,1]^d)] \in (0,\infty)$.
We consider a class of shot-noise fields given by
\begin{equation}\label{eq:shot-noise}
  I_\lambda(z) = \sum_{n=1}^\infty P_n\,\ell(z-X_n),
  \quad z\in\R^d,
\end{equation}
where $P_n$, $n\in\N$, are independent and identically distributed
(i.i.d.)\ nonnegative random variables, called \textit{amplitudes},
which are also independent of $\Phi_\lambda$, and $\ell$ is a
nonnegative and bounded function on $\R^d$, called a \textit{response
  function}.
Shot-noise fields such as \eqref{eq:shot-noise} have been observed in
various areas and studied extensively in the literature
(cf.\ \cite[Sec.~5.6]{ChiuStoyKendMeck13} and references therein).
In particular, recently, they have been used as models for
interference fields in wireless communication networks, where wireless
interferers are located according to a spatial point process
(cf.\ \cite{HaenGant09,BaccBlas09ab}).
In this work, we study the scaling limits of \eqref{eq:shot-noise} as
$\lambda\to\infty$ when $\Phi_\lambda$ is a stationary determinantal
point process.
Determinantal point processes represent a repulsive feature of points
in space and have also been considered as location models of base
stations in cellular wireless
networks~(cf.\ \cite{MiyoShir14a,TorrLeon14,LiBaccDhilAndr15}).

As an early result on the scaling limits of spatial shot-noise fields,
Heinrich and Schmidt~\cite{HeinSchm85} consider a more general form
than \eqref{eq:shot-noise} and show that the centralized and scaled
shot-noise at one position converges in distribution to a Gaussian
random variable when $\Phi_\lambda$ is Brillinger mixing and a
condition which corresponds to $\Exp[{P_n}^2]<\infty$ in
\eqref{eq:shot-noise} is satisfied.
%(see also \cite{HeinPawl13}), 
Since Biscio and Lavancier~\cite{BiscLava16} and
Heinrich~\cite{Hein16} prove that stationary ($\alpha$-)determinantal
point processes are Brillinger mixing, the result
in~\cite{HeinSchm85}, of course, covers the scaling limit of
\eqref{eq:shot-noise} at one position when $\Exp[{P_n}^2]<\infty$.
Based on this background, we here examine the convergence in finite
dimensional distributions and show that the centralized and suitably
scaled version of \eqref{eq:shot-noise} converges to a Gaussian random
field when $\Exp[{P_n}^2]<\infty$ and to an $\alpha$-stable random
field when $P_n$, $n\in\N$, follow a regularly varying distribution
with index~$-\alpha$ for $\alpha\in(1,2)$.
As related work along this direction, Baccelli and
Biswas~\cite{BaccBisw15} consider the shot-noise
field~\eqref{eq:shot-noise}, where $\Phi_\lambda$ is a homogeneous
Poisson point process and the response function~$\ell(x)$ is power-law
and diverges as $x\to0$, and show that a suitably scaled (but
non-centralized) version converges in finite dimensional distributions
to an $\alpha$-stable random field with $\alpha\in(0,1)$.
Aside from this, Kaj~\textit{et al.}~\cite{KajLeskNorrSchm07} consider
a random grain field defined on a homogeneous Poisson point process
associated with a regularly varying volume distribution, and derive
some scaling limits in the sense of convergence in finite dimensional
distributions.
Furthermore, the results in \cite{KajLeskNorrSchm07} are extended by
Breton~\textit{et al.}~\cite{BretClarGoba19} to the one defined on a
stationary determinantal point process.

The rest of the paper is organized as follows.
In the next section, after providing the centralized and scaled
version of \eqref{eq:shot-noise}, we prove the corresponding results
for the shot-noise field defined on a homogeneous Poisson point
process.
These proofs provide the basis for showing our main results, and then
in Section~\ref{sec:determinantal}, we extend them to the one defined
on a stationary determinantal point process.
Finally, conclusion is given in Section~\ref{sec:conclusion}.

%Lebedev~\cite{Lebe03} considers the shot-noise field in
%\eqref{eq:shot-noise} and derive a kind of limit law when
%$\Phi_\lambda$ is a homogeneous Poisson point process and the
%amplitude distribution has a regularly varying tail.
%\cite{Inal12,AljuYani10} %Interference '̐³‹KŠm—¦ê'É'æ'é‹ßŽ—
%\cite{InalHanl10} %Interference '̐³‹KŠm—¦ê'Ö'ÌŽû'©'Ì'¬'³

\section{Preliminary: Scaling limits of Poisson shot-noise
  fields}\label{sec:Preliminaries}

Throughout the paper, we assume that the distribution function~$F_P$
of the amplitudes~$P_n$, $n\in\N$, has the finite mean $p =
\int_0^\infty t\,\dd F_P(t) < \infty$, and the response
function~$\ell$ is bounded and satisfies $\int_{\R^d} \ell(x)\,\dd x <
\infty$.
With this condition, the shot-noise field~$I_\lambda$ in
\eqref{eq:shot-noise} is also stationary on $\R^d$ and Campbell's
formula~(cf.~\cite[p.~8, Theorem~1.2.5]{BaccBlasKarr20}) leads to the
expectation of $I_\lambda(0)$ as
\[
  \Exp[I_\lambda(0)]
  = \lambda p\int_{\R^d} \ell(x)\,\dd x < \infty
\]
(see cf.~\cite{West76} for more general conditions for the almost sure
convergence of shot-noise fields).
The centralized and scaled version of $I_\lambda$ is then given by
\begin{equation}\label{eq:TildeI_lambda}
  \Tilde{I_\lambda}(z)
  = \frac{I_\lambda(z) - \Exp[I_\lambda(z)]}{g(\lambda)},
  \quad z\in\R^d,
\end{equation}
where the function $g$ is suitably chosen, and we investigate its
limit as $\lambda\to\infty$ in the sense of convergence in finite
dimensional distributions when $\Phi_\lambda$ is a determinantal point
process.
As a preliminary, however, we first give the proofs for the case where
$\Phi_\lambda$ is a homogeneous Poisson point process.
%because these proofs provide the basis for proving our main results.

\subsection{Poisson shot-noise with finite second moment of
  amplitudes}

Here, we assume that $\Phi_\lambda$ is a homogeneous Poisson point
process and the amplitude distribution~$F_P$ has the finite second
moment.
The result below is proved straightforwardly and is indeed introduced
without proof in the Introduction of \cite{BaccBisw15}. 
However, we prove it here not only for the completeness of the paper
but also because the proof serves as the basis for showing the later
results.

\begin{proposition}\label{prp:Poi-fv}
Let $\Phi_\lambda = \sum_{n=1}^\infty\delta_{X_n}$ be a homogeneous
Poisson point process with intensity~$\lambda \in (0,\infty)$.
Suppose that $\Exp[{P_1}^2]<\infty$ and let $g(\lambda) =
\lambda^{1/2}$ in \eqref{eq:TildeI_lambda}.
Then, as $\lambda\to\infty$, $\{\Tilde{I_\lambda}(z)\}_{z\in\R^d}$
converges in finite dimensional distributions to a Gaussian random
field~$\{N(z)\}_{z\in\R^d}$ with covariance function;
\begin{align}\label{eq:Poi-fv}
  \Cov[N(z_1), N(z_2)]
  &= \Exp[{P_1}^2]
     \int_{\R^d}\ell(z_1-x)\,\ell(z_2-x)\,\dd x,
   \quad z_1,z_2\in\R^d.
\end{align}
\end{proposition}

Note that the covariance in \eqref{eq:Poi-fv} is finite since $\ell$
is bounded and integrable with respect to the Lebesgue measure on
$\R^d$.
Let $b_\ell$ and $c_\ell$ denote positive constants such that
$\ell(x)\in[0, b_\ell]$ for $x\in\R^d$ and $\int_{\R^d}\ell(x)\,\dd x
= c_\ell$.
Then, we have clearly
\begin{equation}\label{eq:bound}
  \int_{\R^d}\ell(z_1-x)\,\ell(z_2-x)\,\dd x
  \le b_\ell\,c_\ell < \infty.
\end{equation}

\begin{proof}
Consider the finite dimensional Laplace transform of
$\Tilde{I_\lambda}$ in \eqref{eq:TildeI_lambda} at $\bm{z} =
(z_1,\ldots,z_m)\in\R^{d\times m}$ for $m\in\N$; that is, for
$\bm{s}^\top = (s_1,\ldots,s_m)\in[0,\infty)^m$,
\eqref{eq:TildeI_lambda} leads to 
\begin{align}\label{eq:Poi-fv-pr1}
  \Lpl_{\Tilde{I_\lambda}}(\bm{s}, \bm{z})
  &:= \Exp\biggl[
        \exp\biggl(
          - \sum_{j=1}^m s_j\,\Tilde{I_\lambda}(z_j)
        \biggr)
      \biggr]
  \nonumber\\
  &= \Exp\biggl[
       \exp\biggl(
         - \frac{1}{g(\lambda)}
           \sum_{j=1}^m s_j\,I_\lambda(z_j)
       \biggr)
     \biggr]\,
     \exp\biggl(
       \frac{1}{g(\lambda)}
       \sum_{j=1}^m s_j\,\Exp[I_\lambda(z_j)]
     \biggr).
\end{align}
Applying \eqref{eq:shot-noise} and Campbell's formula, we reduce the
inside of the second exponential in the last expression above to
\begin{equation}\label{eq:Poi-fv-pr2}
  \frac{1}{g(\lambda)}
  \sum_{j=1}^m s_j\,\Exp[I_\lambda(z_j)]
  = \frac{\lambda\,p}{g(\lambda)}
    \int_{\R^d} \xi_{\bm{s},\bm{z}}(x)\,\dd x,
\end{equation}
where $\xi_{\bm{s},\bm{z}}(x) := \sum_{j=1}^m s_j\,\ell(z_j-x)$ and
$\Exp[P_n]=p$ is used.
On the other hand, applying \eqref{eq:shot-noise} and the
probability generating functional of a Poisson point process
(cf.~\cite[p.~25, Exercise~3.6]{LastPenr17}) to the first expectation
in the last expression of \eqref{eq:Poi-fv-pr1} leads to
\begin{align}\label{eq:Poi-fv-pr3}
  \Exp\biggl[
    \exp\biggl(
      - \frac{1}{g(\lambda)}
        \sum_{j=1}^m s_j\,I_\lambda(z_j)
    \biggr)
  \biggr]
  &= \Exp\biggl[
       \prod_{n=1}^\infty
         \Lpl_P\Bigl(\frac{\xi_{\bm{s},\bm{z}}(X_n)}{g(\lambda)}\Bigr)
     \biggr]
  \nonumber\\
  &= \exp\biggl(
       - \lambda
         \int_{\R^d}
           \Bigl[
             1 - \Lpl_P\Bigl(
                   \frac{\xi_{\bm{s},\bm{z}}(x)}{g(\lambda)}
                 \Bigr)
           \Bigr]\,
         \dd x
     \biggr),
\end{align}
where $\Lpl_P(s)=\Exp[e^{-s P_1}]$ denotes the Laplace transform of
$P_1$.
Therefore, plugging \eqref{eq:Poi-fv-pr2} and \eqref{eq:Poi-fv-pr3}
into \eqref{eq:Poi-fv-pr1}, we have
\begin{equation}\label{eq:Poi-fv-pr4}
  \Lpl_{\Tilde{I_\lambda}}(\bm{s}, \bm{z})
  = \exp\biggl(
      \lambda
      \int_{\R^d}\!\int_0^\infty
        \psi\Bigl(
          \frac{\xi_{\bm{s},\bm{z}}(x)}{g(\lambda)}\,t
        \Bigr)\,
      \dd F_P(t)\,\dd x
    \biggr),
\end{equation}
where $\psi(u) = e^{-u}-1+u$, and we used $\Lpl_P(s) = \int_0^\infty
e^{-s t}\,\dd F_P(t)$ and $p = \int_0^\infty t\,\dd F_P(t)$.
We now take $g(\lambda) = \lambda^{1/2}$.
Then, since $\psi(u) = u^2/2 + o(u^2)$ as $u\downarrow0$, if the order
of the limit and the integral is interchangeable (which is confirmed
below), we obtain
\begin{align}\label{eq:Poi-fv-pr5}
  \lim_{\lambda\to\infty}
    \Lpl_{\Tilde{I_\lambda}}(\bm{s}, \bm{z})
  &= \exp\biggl(
       \frac{\Exp[{P_1}^2]}{2}
       \int_{\R^d}[\xi_{\bm{s},\bm{z}}(x)]^2\,\dd x
     \biggr)
  \nonumber\\     
  &= \exp\biggl(
       \frac{\Exp[{P_1}^2]}{2}\,
       \bm{s}^\top\bm{L}(\bm{z})\,\bm{s}
     \biggr),    
\end{align}
where the $(j,k)$-element of matrix~$\bm{L}(\bm{z}) =
\bigl(L(z_j,z_k)\bigr)_{j,k=1}^m$ is given by $L(z_j, z_k) =
\int_{\R^d}\ell(z_j-x)\,\ell(z_k-x)\,\dd x$ and the assertion of the
proposition holds.
It remains to confirm the interchangeability of the order of the limit
and the integral in \eqref{eq:Poi-fv-pr4} as $\lambda\to\infty$.
Since $\psi(u) \in[0, u^2/2]$, we have $0\le
\lambda\,\psi\bigl(\xi_{\bm{s},\bm{z}}(x)\,t/\lambda^{1/2}\bigr) \le
[\xi_{\bm{s},\bm{z}}(x)\,t]^2/2$ and the integral of
$[\xi_{\bm{s},\bm{z}}(x)\,t]^2/2$ with respect to $\dd F_P(t)\,\dd x$
is provided as the inside of the exponential in \eqref{eq:Poi-fv-pr5},
which is finite from \eqref{eq:bound}.
Hence, the dominated convergence theorem is applicable and the proof
is completed.
\end{proof}

\subsection{Poisson shot-noise with regularly varying amplitude
  distribution}

Next, we assume that the tail $\Bar{F_P}(t)=1-F_P(t)$ of the amplitude
distribution is regularly varying with index~$-\alpha$ for
$\alpha\in(1,2)$; that is (cf.~\cite{BingGoldTeug87} or
\cite{Sene76}),
\[
  \lim_{t\to\infty}\frac{\Bar{F_P}(c t)}{\Bar{F_P}(t)}
  = c^{-\alpha}
  \quad\text{for some $c>0$.}
\]
Note that $\Exp[{P_i}^2]=\infty$ in this case.
We use the following properties of regularly varying functions, where
$a(x)\sim b(x)$ as $x\to\infty$ stands for
$\lim_{x\to\infty}a(x)/b(x)=1$.

\begin{proposition}\label{prp:Karamata}
%In the following, $a(x)\sim b(x)$ as $x\to\infty$ stands for
%$\lim_{x\to\infty}a(x)/b(x)=1$.
\begin{enumerate}%[(i)]  
\item (Cf.~\cite[p.~28, Theorem~1.5.12]{BingGoldTeug87}
  or~\cite[p.~21]{Sene76})
  Let $f$ be regularly varying with index~$\gamma>0$.
  Then, there exists an asymptotic inverse~$g$ of $f$ satisfying
  $f(g(x)) \sim g(f(x)) \sim x$ as $x\to\infty$, where $g$ is
  asymptotically unique and regularly varying with index
  $1/\gamma$.
\item (Representation Theorem; cf.~\cite[p.~12,
  Theorem~1.3.1]{BingGoldTeug87} or~\cite[p.~2, Theorem~1.2]{Sene76}
  Function~$L_0$ is slowly varying if and only if there exists a
  positive constant~$a_0$ such that
  \begin{equation}\label{eq:Representation}
    L_0(x)
    = \exp\Bigl(
        \eta (x) + \int_{a_0}^x\frac{\epsilon(t)}{t}\,\dd t
      \Bigr),
    \quad x\ge a_0,    
  \end{equation}
  where $\eta(x)$ is bounded and converges to a constant as
  $x\to\infty$, and $\epsilon(t)$ is bounded and converges to zero as
  $t\to\infty$.
%\item (Karamata's Theorem; cf.~\cite[p.~28,
%  Theorem~1.5.11]{BingGoldTeug87})
%  Let $f$ be regularly varying with index~$\gamma\in\R$ and be locally
%  bounded on $[\delta,\infty)$ with some $\delta>0$.
%  Then, for $\sigma \ge -(\gamma+1)$,
%  \begin{equation}\label{eq:Karamata1}
%    \int_\delta^x t^\sigma\,f(t)\,\dd t
%    \sim \frac{x^{\sigma+1}\,f(x)}{\sigma + \gamma + 1}
%    \quad\text{as $x\to\infty$,}
%  \end{equation}
%  and for $\sigma< - (\gamma+1)$,
%  \begin{equation}\label{eq:Karamata2}
%    \int_x^\infty t^\sigma\,f(t)\,\dd t
%    \sim - \frac{x^{\sigma+1}\,f(x)}{\sigma + \gamma + 1}
%    \quad\text{as $x\to\infty$.}
%  \end{equation}
\end{enumerate}
\end{proposition}

Using the properties above, we prove the following result, which is
also new to the best of the knowledge of the authors.

\begin{theorem}\label{thm:Poi-infv}
Let $\Phi_\lambda = \sum_{i=1}^\infty\delta_{X_i}$ be a homogeneous
Poisson point process with intensity~$\lambda \in (0,\infty)$.
Suppose that $\Bar{F_P}$ is regularly varying with index~$-\alpha$ for
$\alpha\in(1,2)$ and let $g$ in \eqref{eq:TildeI_lambda} be an
asymptotic inverse of $1/\Bar{F_P}$ (so that $g$ is regularly varying
with index $1/\alpha$).
Then, as $\lambda\to\infty$, $\{\Tilde{I_\lambda}(y)\}_{y\in\R^d}$
converges in finite dimensional distributions to an $\alpha$-stable
random field~$\{S(z)\}_{z\in\R^d}$ with finite dimensional Laplace
transform;
\begin{align}\label{eq:Poi-infv}
  \Lpl_S(\bm{s}, \bm{z})
  &:= \Exp\biggl[
        \exp\biggl(
          - \sum_{j=1}^m s_j\,S(z_j)
        \biggr)
      \biggr]   
  \nonumber\\
  &= \exp\biggl(
       \frac{\Gamma(2-\alpha)}{\alpha-1}
       \int_{\R^d}
         \biggl[\sum_{j=1}^m s_j\,\ell(z_j-x)\biggr]^\alpha\,
       \dd x
     \biggr),
\end{align}
for $m\in\N$, $\bm{s}^\top = (s_1,\ldots,s_m)\in[0,\infty)^m$ and
$\bm{z}=(z_1,\ldots,z_m)\in\R^{d\times m}$.
\end{theorem}

\begin{remark}
The integral in \eqref{eq:Poi-infv} is finite for any fixed
$\bm{s}^\top = (s_1,\ldots,s_m)\in[0,\infty)^m$ and
$\bm{z}=(z_1,\ldots,z_m)\in\R^{d\times m}$.
Indeed, since $\alpha>1$, it is easy to show that
\[
  \int_{\R^d}\biggl[\sum_{j=1}^m s_j\,\ell(z_j-x)\biggr]^\alpha\,\dd x
  \le {b_\ell}^{\alpha-1}\,c_\ell\,\biggl(\sum_{j=1}^m s_j\biggr)^\alpha
  < \infty,
\]
where $b_\ell$ and $c_\ell$ are the same as in \eqref{eq:bound}.
The last expression of \eqref{eq:Poi-infv} definitely implies that
$\{S(z)\}_{z\in\R^d}$ is an $\alpha$-stable random field since each
linear combination~$\sum_{j=1}^m s_j\,S(z_j)$ for $m\in\N$,
$\bm{s}^\top = (s_1,\ldots,s_m)\in[0,\infty)^m$ and
$\bm{z}=(z_1,\ldots,z_m)\in\R^{d\times m}$ follows an $\alpha$-stable
distribution~$\mathcal{S}_\alpha(\sigma_{\bm{s},\bm{z}},1,0)$ with
\[
  \sigma_{\bm{s},\bm{z}}
  = \biggl(
      - \frac{\Gamma(2-\alpha)}{\alpha-1}
        \int_{\R^d}
          \biggl[\sum_{j=1}^m s_j\,\ell(z_j-x)\biggr]^\alpha\,
        \dd x\,
        \cos\frac{\pi\alpha}{2}
    \biggr)^{1/\alpha}
\]  
(cf.~\cite[p.~15, Proposition~1.2.12 and pp.~112--113,
    Theorem~3.1.2]{SamoTaqq94}). 
\end{remark}

\begin{proof}
We start the proof of the theorem with \eqref{eq:Poi-fv-pr4} in the
proof of Proposition~\ref{prp:Poi-fv}, where we recall that
$\xi_{\bm{s},\bm{z}}(x) = \sum_{j=1}^m s_j\,\ell(z_j-x)$ and $\psi(u)
= e^{-u}-1+u$.
Applying integration by parts to the integral with respect to
$\dd F_P(t)$ in \eqref{eq:Poi-fv-pr4}, we have
\begin{align*}
  \int_0^\infty
    \psi\biggl(
      \frac{\xi_{\bm{s},\bm{z}}(x)}{g(\lambda)}\,t
    \biggr)\,
  \dd F_P(t)
  &= \frac{\xi_{\bm{s},\bm{z}}(x)}{g(\lambda)}
     \int_0^\infty
       \Bigl[
         1 - \exp\Bigl(- \frac{\xi_{\bm{s},\bm{z}}(x)}
                              {g(\lambda)}\,t\Bigr)
       \Bigr]\,
       \Bar{F_P}(t)\,
     \dd t
  \\
  &= \xi_{\bm{s},\bm{z}}(x)
     \int_0^\infty
       \bigl[1 - e^{-\xi_{\bm{s},\bm{z}}(x)\,u}\bigr]\,
       \Bar{F_P}\bigl(g(\lambda)\,u\bigr)\,
     \dd u,
\end{align*}
where the change of variables $u = t/g(\lambda)$ is applied in the
second equality.
Therefore, since $\lambda \sim 1/\Bar{F_P}(g(\lambda))$ and
$\Bar{F_P}(g(\lambda)\,u)/\Bar{F_P}(g(\lambda)) \to u^{-\alpha}$ as
$\lambda\to\infty$, if the order of the limit and the integral is
interchangeable (which is confirmed below), the inside of the
exponential in \eqref{eq:Poi-fv-pr4} yields
\begin{align}\label{eq:Poi-infv-pr2}
  &\lambda\int_{\R^d}\!\int_0^\infty
     \psi\Bigl(\frac{\xi_{\bm{s},\bm{z}}(x)}{g(\lambda)}\,t\Bigr)\,
   \dd F_P(t)\,\dd x
  \nonumber\\   
  &\sim \int_{\R^d} \xi_{\bm{s},\bm{z}}(x)
          \int_0^\infty
            \bigl[1 - e^{-\xi_{\bm{s},\bm{z}}(x)u}\bigr]\,
            \frac{\Bar{F_P}(g(\lambda)\,u)}{\Bar{F_P}(g(\lambda))}\,
          \dd u\,\dd x
  \nonumber\\
  &\to \int_{\R^d}[\xi_{\bm{s},\bm{z}}(x)]^\alpha\,\dd x
         \int_0^\infty (1 - e^{-v})\,v^{-\alpha}\,\dd v
   \quad\text{as $\lambda\to\infty$,}
\end{align}
where $v=\xi_{\bm{s},\bm{z}}(x)\,u$ is used in the last expression.
The last integral above is equal to $\Gamma(2-\alpha)/(\alpha-1)$ and
the last expression of \eqref{eq:Poi-infv} is obtained.

It remains to show the interchangeability of the order of the limit
and the integral in \eqref{eq:Poi-infv-pr2}, where the dominated
convergence theorem can be applied if we can find an integrable bound
on
\begin{equation}\label{eq:Poi-infv-pr3}
  \xi(x)\,[1-e^{-\xi(x)u}]\,
  \frac{\Bar{F_P}(g u)}{\Bar{F_P}(g)},
\end{equation}
with respect to $\dd u\,\dd x$ on $[0,\infty)\times\R^d$ for a
positive and integrable function~$\xi$ on $\R^d$ and a sufficiently
large $g>0$.
Since $\Bar{F_P}$ is regularly varying with index $-\alpha$, we have
$\Bar{F_P}(g) = g^{-\alpha}\,L_0(g)$ with $L_0$ of the
form~\eqref{eq:Representation}.
We define constants~$\eta^*$ and $\epsilon^*$ as
\[
  \eta^* = \sup_{x\ge a_0}|\eta(x)|,
  \quad
  \epsilon^* = \sup_{t\ge a_0}|\epsilon(t)|.
\]
Note here that we can take $a_0$ in \eqref{eq:Representation} large
enough such that $\epsilon^* < \alpha-1$ since $\epsilon(t)\to0$ as
$t\to\infty$.
Then, for $g\ge a_0$ and $u\ge 1$, we have
\[
  \frac{\Bar{F_P}(gu)}{\Bar{F_P}(g)}
  \le u^{-\alpha}\,
      \exp\Bigl(
        2\eta^* + \epsilon^*\int_g^{gu}\frac{\dd t}{t}
      \Bigr)
  = e^{2\eta^*} u^{-(\alpha-\epsilon^*)},
\]
and \eqref{eq:Poi-infv-pr3} is bounded by
\[
  \xi(x)\,\bigl(
    b_0\,\ind{(0,1)}(u)
    + e^{2\eta^*} u^{-(\alpha-\epsilon^*)}\,\ind{[1,\infty)}(u)
  \bigr),
\]
where $b_0= \sup_{g\ge a_0,u\in(0,1)}\Bar{F_P}(g u)/\Bar{F_P}(g)$.
%Since $1-e^{-\xi(x)u} \le \xi(x)\,u\,\ind{(0,1]}(u) +
%\ind{(1,\infty)}(u)$ using the indicator function~$\ind{A}$ of a
%set~$A$, \eqref{eq:Poi-infv-pr3} is bounded by
%\[
%  \xi(x)^2\,
%  u\,\frac{\Bar{F_P}(g u)}{\Bar{F_P}(g)}\,\ind{(0,1]}(u)
%  + \xi(x)\,
%    \frac{\Bar{F_P}(g u)}{\Bar{F_P}(g)}\,\ind{(1,\infty)}(u).
%\]
We know that $\xi$ is integrable on $\R^d$ and
\[
  \int_1^\infty u^{-(\alpha-\epsilon^*)}\,\dd u
  = \frac{1}{\alpha-1-\epsilon^*},
\]
%, it suffices to show the
%integrability of the above with respect to $\dd u$ on $[0,\infty)$.
%First, applying \eqref{eq:Karamata1} in Proposition~\ref{prp:Karamata}
%with $\gamma=-\alpha$ and $\sigma=1$, we have for any $\epsilon>0$ and
%a sufficiently large~$g>0$,
%\begin{align*}
%  \int_0^1 u\,\frac{\Bar{F_P}(g u)}{\Bar{F_P}(g)}\,\dd u
%  &= \frac{1}{g^2\,\Bar{F_P}(g)}
%     \int_0^g v\,\Bar{F_P}(v)\,\dd v
%  \\
%  &\le \frac{1+\epsilon}{2-\alpha} < \infty.
%\end{align*}
%Similarly, applying \eqref{eq:Karamata2} with $\gamma=-\alpha$ and
%$\sigma=0$, we have
%\begin{align*}
%  \int_1^\infty \frac{\Bar{F_P}(g u)}{\Bar{F_P}(g)}\,\dd u
%  &= \frac{1}{g\,\Bar{F_P}(g)}
%     \int_g^\infty \Bar{F_P}(v)\,\dd v
%  \\
%  &\le \frac{1+\epsilon}{\alpha-1} < \infty,
%\end{align*}
which completes the proof.
\end{proof}

\section{Scaling limits of determinantal shot-noise
  fields}\label{sec:determinantal}

We now extend the results in the preceding section to the case where
$\Phi_\lambda$ is a stationary determinantal point process and show
that the same scaling limits are derived.
Let $\Phi_\lambda = \sum_{n=1}^\infty\delta_{X_n}$ be a stationary and
isotropic determinantal point process on $\R^d$ with
intensity~$\lambda\in(0,\infty)$ and let
$K_\lambda$:~$\R^d\times\R^d\to\C$ denote the kernel of $\Phi_\lambda$
with respect to the Lebesgue measure; that is, the $n$th product
density~$\rho_n$, $n\in\N$, of $\Phi_\lambda$ with respect to the
Lebesgue measure is given by (cf.\ \cite{HougKrisPereVira09,Sosh00})
\[
  \rho_n(x_1,x_2,\ldots,x_n)
  = \det\bigl(K_\lambda(x_i,x_j)\bigr)_{i,j=1,2,\ldots,n},
  \quad x_1,x_2,\ldots,x_n\in\R^d,
\]
where $\det$ stands for determinant.
We assume that (i)~the kernel~$K_\lambda$ is continuous on
$\R^d\times\R^d$ with $K_\lambda(x,x) = \rho_1(x) = \lambda$ for
any $x\in\R^d$; (ii)~$K_\lambda$ is Hermitian; that is,
$K_\lambda(x,y) = K_\lambda(y,x)^*$ for $x,y\in\R^d$, where $w^*$
denotes the complex conjugate of $w\in\C$; and (iii)~the integral
operator~$\K_\lambda$ on $L^2(\R^d,\dd x)$ given by
\[
  \K_\lambda f(x)
  = \int_{\R^d} K_\lambda(x,y)\,f(y)\,\dd y,
  \quad f\in L^2(\R^d,\dd x),\; x\in\R^d,
\]
has its spectrum in $[0,1]$.
Note that the operator~$\K_\lambda$ satisfying (i)--(iii) is locally
of trace-class (cf.~\cite[p.~65, Lemma]{ReedSimoIII79}), and that the
determinantal point process~$\Phi_\lambda$ exists and is locally
finite~(cf.~\cite[p.~68, Theorem~4.5.5]{HougKrisPereVira09} or
\cite[Theorem~3]{Sosh00}).
Moreover, we assume that $K_\lambda$ satisfies $|K_\lambda(x,y)|^2
= |K_\lambda(0,y-x)|^2$ which depends only on the distance~$\|x-y\|$ of
$x,y\in\R^d$.
The product densities~$\rho_n$, $n\in\N$, are then motion-invariant
(invariant to translations and rotations) and $\rho_2(0,x) = \lambda^2
- |K_\lambda(0,x)|^2$ depends only on $\|x\|$ for $x\in\R^d$.

To develop the corresponding discussion to the case of a Poisson point
process, we first give a preliminary lemma.
%and to apply a similar technique to \cite{BretClarGoba19},

\begin{lemma}\label{lem:det-fdl}
Let $\Phi_\lambda$ be the determinantal point process described above.
Then, the finite dimensional Laplace transform of the shot-noise
field~$I_\lambda$ in~\eqref{eq:shot-noise} has the following
exponential expression; 
\begin{align}\label{eq:det-fdl}
  \Lpl_{I_\lambda}(\bm{s},\bm{z})
  &:= \Exp\biggl[
        \exp\biggl(- \sum_{j=1}^m s_j\,I_\lambda(z_j)\biggr)
      \biggr]
  \nonumber\\
  &= \exp\biggl(
       - \sum_{n=1}^\infty
           \frac{1}{n}\,\Tr\bigl({\K_{\lambda,\Lpl\circ\xi}}^n\bigr)
     \biggr),
\end{align}
for $m\in\N$, $\bm{z} = (z_1,\ldots,z_m)\in\R^{d\times m}$ and
$\bm{s}^\top = (s_1,\ldots,s_m)\in[0,\infty)^m$, where $\Tr$ stands
for the trace of a linear operator and $\K_{\lambda,\Lpl\circ\xi}$
denotes the integral operator given by the kernel;
\begin{equation}\label{eq:K_(lambda,L)}
  K_{\lambda,\Lpl\circ\xi}(x,y)
  = \sqrt{1-\Lpl_P\bigl(\xi_{\bm{s},\bm{z}}(x)\bigr)}
    \,K_\lambda(x,y)\,
    \sqrt{1-\Lpl_P\bigl(\xi_{\bm{s},\bm{z}}(y)\bigr)},
  \quad x,y\in\R^d,
\end{equation}
with the Laplace transform~$\Lpl_P$ of $P_1$ and
$\xi_{\bm{s},\bm{z}}(x) = \sum_{j=1}^m s_j\,\ell(z_j-x)$.
\end{lemma}

To prove the lemma, we use the following result in the literature.

\begin{proposition}[{Cf.~\cite[Theorem~1.2]{ShirTaka03} and
  \cite[Lemma~2 and
    Corollary~1]{LiBaccDhilAndr15}}]\label{prp:det-pgfl}
Let $\Phi=\sum_{n=1}^\infty\delta_{X_n}$ denote a determinantal point
process on $\R^d$, where the kernel~$K$ with respect to the Lebesgue
measure ensures the existence of $\Phi$.
Then, for any measurable function~$v$:~$\R^d\to[0,1]$ such that $f(x)
:= -\ln v(x)$ satisfies (a)~$\lim_{\|x\|\to\infty}f(x)=0$,
(b)~$\lim_{r\to\infty}\int_{\|x\|>r}K(x,x)\,f(x)\,\dd x = 0$, and
(c)~$\int_{\R^d}K(x,x)\bigl[1-\exp\bigl(-f(x)\bigr)\bigr]\,\dd x <
\infty$, the probability generating functional of $\Phi$ is given by
\[
  \Exp\biggl[\prod_{n=1}^\infty v(X_n)\biggr]
  = \Det(\mathcal{I}-\K_v),
\]
where $\Det$ stands for the Fredholm determinant, $\mathcal{I}$
denotes the identity operator and $\K_v$ is the integral operator
given by the kernel~$K_v(x,y) = \sqrt{1-v(x)}\,K(x,y)\sqrt{1-v(y)}$,
$x,y\in\R^d$.
\end{proposition}

The result in Proposition~\ref{prp:det-pgfl} is first presented in
\cite{ShirTaka03} in the form of Laplace functional for
function~$f(x)=-\ln v(x)$ such that $f$ has a compact support.
It is then generalized in \cite{LiBaccDhilAndr15} to $f$ satisfying
the conditions~(a)--(c) in the proposition when $d=2$, whereas this
generalization is also available for $\R^d$, $d=2,3,\ldots$.

\begin{proof}[Proof of Lemma~\ref{lem:det-fdl}]
Similar to obtaining the first equality in \eqref{eq:Poi-fv-pr3}, we
have
\begin{equation}\label{eq:det-fdl-pr1}
  \Lpl_{I_\lambda}(\bm{s},\bm{z})
  = \Exp\biggl[
      \prod_{n=1}^\infty\Lpl_P\bigl(\xi_{\bm{s},\bm{z}}(X_n)\bigr)
    \biggr].
\end{equation}
To apply Proposition~\ref{prp:det-pgfl}, we have to confirm that
$f(x)=-\ln\Lpl_P\bigl(\xi_{\bm{s},\bm{z}}(x)\bigr)$ satisfies
conditions~(a)--(c) in it.
For (a), recall that $\xi_{\bm{s},\bm{z}}(x) = \sum_{j=1}^m
s_j\,\ell(z_j-x)$ and $\Lpl_P\bigl(\xi_{\bm{s},\bm{z}}(x)\bigr) =
\Exp[e^{-\xi_{\bm{s},\bm{z}}(x)\,P_1}]$.
Since $e^{-\xi_{\bm{s},\bm{z}}(x)\,P_1}\in(0,1]$ and
$e^{-\xi_{\bm{s},\bm{z}}(x)\,P_1}\to1$ as $\|x\|\to\infty$, the dominated
convergence theorem leads to $-\ln\Lpl_P\bigl(\xi_{\bm{s},\bm{z}}(x)\bigr)\to0$
as $\|x\|\to\infty$.
Next, we confirm (b).
Since $K_\lambda(x,x)=\lambda$, it suffices to show that
$\int_{\R^d}\bigl[-\ln\Lpl_P\bigl(\xi_{\bm{s},\bm{z}}(x)\bigr)\bigr]\,\dd x <
\infty$, which follows from the integrability of $\xi_{\bm{s},\bm{z}}$ because
$-\ln\Lpl_P\bigl(\xi_{\bm{s},\bm{z}}(x)\bigr) \le p\,\xi_{\bm{s},\bm{z}}(x)$ by Jensen's
inequality.
The condition~(c) is confirmed by showing $\int_{\R^d} \bigl[1 -
  \Lpl_P\bigl(\xi_{\bm{s},\bm{z}}(x)\bigr)\bigr]\,\dd x < \infty$.
Integration by parts yields
\begin{align*}
  1 - \Lpl_P\bigl(\xi_{\bm{s},\bm{z}}(x)\bigr)
  &= \int_0^\infty [1 - e^{-\xi_{\bm{s},\bm{z}}(x)\,t}]\,\dd F_P(t)
  \\
  &= \xi_{\bm{s},\bm{z}}(x)
     \int_0^\infty e^{-\xi_{\bm{s},\bm{z}}(x)\,t}\,\Bar{F_P}(t)\,\dd t
 \le p\,\xi_{\bm{s},\bm{z}}(x),
\end{align*}
which is integrable.
Therefore, we can apply Proposition~\ref{prp:det-pgfl} to
\eqref{eq:det-fdl-pr1} and obtain
\[
  \Lpl_{I_\lambda}(\bm{s},\bm{z})
  = \Det(\mathcal{I}-\K_{\lambda,\Lpl\circ\xi}).
\]
By the condition~(c) above, the operator $\K_{\lambda,\Lpl\circ\xi}$
given by \eqref{eq:K_(lambda,L)} is of trace-class.
Moreover, its operator norm satisfies
$\|\K_{\lambda,\Lpl\circ\xi}\|_{\op} < \|\K_\lambda\|_{\op} \le 1$ since
$\Lpl_P\bigl(\xi_{\bm{s},\bm{z}}(x)\bigr)$ is strictly positive.
Hence, the Fredholm determinant
$\mathsf{Det}(\mathcal{I}-\mathcal{K}_{\lambda,\Lpl\circ\xi})$ has the
exponential expression~\eqref{eq:det-fdl} (cf.~\cite[p.~331,
  Lemma~6]{ReedSimoIV78}).
\end{proof}

\subsection{Case of finite second moment of amplitudes}

Here is the extension of Proposition~\ref{prp:Poi-fv} to the case of a
determinantal point process, which we prove by applying a similar
discussion to that in~\cite{BretClarGoba19}.

\begin{theorem}\label{thm:det-fv}
Let $\Phi_\lambda$ be the determinantal point process with
intensity~$\lambda$ described above.
Suppose that $\Exp[{P_1}^2]<\infty$ and $g(\lambda) = \lambda^{1/2}$
in \eqref{eq:TildeI_lambda}.
In addition, we assume that
\begin{equation}\label{eq:AssumptionA}
  \int_{\R^d}|K_\lambda(0,x)|^2\,\dd x = o(\lambda)
  \quad\text{as $\lambda\to\infty$.}
\end{equation}
Then, as $\lambda\to\infty$, $\{\Tilde{I_\lambda}(z)\}_{z\in\R^d}$
converges in finite dimensional distributions to the same Gaussian
random field as in Proposition~\ref{prp:Poi-fv}.
\end{theorem}

\begin{remark}\label{rem:MiyoShir17}
In general, it holds that $\int_{\R^d}|K_\lambda(0,x)|^2\,\dd x \le
K_\lambda(0,0) = \lambda$ (see \cite[Lemma~3.3]{MiyoShir17}).
Since $\rho_2(0,x) = \lambda^2 - |K_\lambda(0,x)|^2$,
condition~\eqref{eq:AssumptionA} above requires that the negative
correlation in $\Phi_\lambda$ is weakening as $\lambda\to\infty$.
%and becomes close to a Poisson point process 
\end{remark}

\begin{proof}
Similar to obtaining \eqref{eq:Poi-fv-pr1}, \eqref{eq:Poi-fv-pr2} and
\eqref{eq:Poi-fv-pr3}, we have
\begin{equation}\label{eq:det-fv-pr1}
  \Lpl_{\Tilde{I_\lambda}}(\bm{s}, \bm{z})
  = \Exp\biggl[
      \prod_{n=1}^\infty
        \Lpl_P\Bigl(\frac{\xi_{\bm{s},\bm{z}}(X_n)}{g(\lambda)}\Bigr)
    \biggr]
    \,
    \exp\biggl(
      \frac{\lambda\,p}{g(\lambda)}
      \int_{\R^d} \xi_{\bm{s},\bm{z}}(x)\,\dd x
    \biggr),
\end{equation}
where we recall that $\xi_{\bm{s},\bm{z}}(x) = \sum_{j=1}^m
s_j\,\ell(z_j-x)$.
By Lemma~\ref{lem:det-fdl}, the expectation on the right-hand side
above is equal to
\begin{equation}\label{eq:det-fv-pr2}
  \Exp\biggl[
    \prod_{n=1}^\infty
      \Lpl_P\biggl(
        \frac{\xi_{\bm{s}, \bm{z}}(X_n)}{g(\lambda)}
      \biggr)
  \biggr]
  = \exp\biggl(
     - \sum_{n=1}^\infty
         \frac{1}{n}\,
         \Tr\bigl({\K_{\lambda,\Lpl\circ(\xi/g)}}^n\bigr)
   \biggr),
\end{equation}
where $\K_{\lambda,\Lpl\circ(\xi/g)}$ denotes the integral operator
given by the kernel;
\[
  K_{\lambda,\Lpl\circ(\xi/g)}(x,y)
  = \sqrt{1-\Lpl_P\Bigl(\frac{\xi_{\bm{s},\bm{z}}(x)}{g(\lambda)}\Bigr)}\,
    K_\lambda(x,y)\,
    \sqrt{1-\Lpl_P\Bigl(\frac{\xi_{\bm{s},\bm{z}}(y)}{g(\lambda)}\Bigr)}.
\]
Note that the term of $n=1$ inside the exponential
in~\eqref{eq:det-fv-pr2} is equal to
\[
 \Tr\bigl(\K_{\lambda,\Lpl\circ(\xi/g)}\bigr)
 = \lambda
   \int_{\R^d}
     \Bigl[
       1 - \Lpl_P\Bigl(\frac{\xi_{\bm{s},\bm{z}}(x)}{g(\lambda)}\Bigr)
     \Bigr]\,
   \dd x,
\]
which is identical to \eqref{eq:Poi-fv-pr3} in the case of a Poisson
point process.
Therefore, the proof is completed if we can show that
\begin{equation}\label{eq:det-fv-pr3}
  \sum_{n=2}^\infty
    \frac{1}{n}\,
    \Tr\bigl({\K_{\lambda,\Lpl\circ(\xi/g)}}^n\bigr)
  \to 0
  \quad\text{as $\lambda\to\infty$.}
\end{equation}
Note that it holds that $\Tr(|A|^n) \le
\Tr(|A|^2)^{1/2}\,\Tr(|A|^{n-1})$ for a trace-class operator~$A$ since
$\Tr(|AB|) \le \|A\|_{\op}\,\Tr(|B|)$ for a bounded operator~$A$ and a
trace-class operator~$B$ (cf.~\cite[p.~218, Problem~28]{ReedSimoI80})
and that $\|A\|_{\op} \le \Tr(|A|^2)^{1/2}$ for a Hilbert-Schmidt
operator~$A$ (cf.~\cite[p.~210, Theorem~VI.22 (d) or p.~218,
  Problem~25]{ReedSimoI80}).
Applying this to the left-hand side of \eqref{eq:det-fv-pr3}
inductively, we have
\begin{align}\label{eq:det-fv-pr4}
  \biggl|
    \sum_{n=2}^\infty
      \frac{1}{n}\,
      \Tr\bigl({\K_{\lambda,\Lpl\circ(\xi/g)}}^n\bigr)
  \biggr|
  &\le \sum_{n=2}^\infty
         \frac{1}{n}\,
         \Tr\bigl(|\K_{\lambda,\Lpl\circ(\xi/g)}|^n\bigr)
  \nonumber\\
  &\le \sum_{n=1}^\infty
         \frac{1}{n}\,
         \bigl(
           \Tr\bigl(|\K_{\lambda,\Lpl\circ(\xi/g)}|^2\bigr)
         \bigr)^{n/2}
  \nonumber\\
  &= - \ln\Bigl(
         1 - \bigl(
               \Tr\bigl(|\K_{\lambda,\Lpl\circ(\xi/g)}|^2\bigr)
             \bigr)^{1/2}
       \Bigr),
\end{align}
where the last equality holds when
$\bigl(\Tr\bigl(|\K_{\lambda,\Lpl\circ(\xi/g)}|^2\bigr)\bigr)^{1/2}<1$,
which is ensured for sufficiently large~$\lambda$ as shown below.
Since $1-\Lpl_P(s) \le s p$ and $|K_\lambda(x,y)|^2 =
|K_\lambda(0,y-x)|^2$,
\begin{align}\label{eq:det-fv-pr5}
  \Tr\bigl(|\K_{\lambda,\Lpl\circ(\xi/g)}|^2\bigr)
  &= \int_{\R^d}\!\int_{\R^d}
       \Bigl(1-\Lpl_P\Bigl(\frac{\xi_{\bm{s},\bm{z}}(x)}{g(\lambda)}\Bigr)\Bigr)\,
       |K_\lambda(x,y)|^2\,
       \Bigl(1-\Lpl_P\Bigl(\frac{\xi_{\bm{s},\bm{z}}(y)}{g(\lambda)}\Bigr)\Bigr)\,
     \dd x\,\dd y
  \nonumber\\
  &\le \Bigl(\frac{p}{g(\lambda)}\Bigr)^2
       \int_{\R^d}\!\int_{\R^d}
         \xi_{\bm{s},\bm{z}}(x)|K_\lambda(x,y)|^2\,\xi_{\bm{s},\bm{z}}(y)\,
       \dd x\,\dd y
  \nonumber\\
  &\le b_\ell\,c_\ell\,
       \biggl(\sum_{j=1}^m s_j\biggr)^2\,
       \Bigl(\frac{p}{g(\lambda)}\Bigr)^2
       \int_{\R^d} |K_\lambda(0,y)|^2\,\dd y,
\end{align}
where $b_\ell$ and $c_\ell$ are the same as in \eqref{eq:bound}, and
we use $\xi_{\bm{s},\bm{z}}(y) = \sum_{j=1}^m s_j\,\ell(z_j-y) \le
b_\ell\sum_{j=1}^m s_j$ and $\int_{\R^d} \xi_{\bm{s},\bm{z}}(x)\,\dd x
= c_\ell\sum_{j=1}^m s_j$ for fixed $\bm{s}\in[0,\infty)^m$ and
$\bm{z}\in\R^{d\times m}$.
Hence, when $g(\lambda) = \lambda^{1/2}$, \eqref{eq:det-fv-pr5} and
therefore \eqref{eq:det-fv-pr4} go to $0$ as $\lambda\to\infty$ under
assumption~\eqref{eq:AssumptionA}, which implies \eqref{eq:det-fv-pr3}.
\end{proof}

\subsection{Case of regularly varying amplitude distribution}

Here is our final result in this work, which is the extension of
Theorem~\ref{thm:Poi-infv} to the case of a determinantal point
process.

\begin{theorem}\label{thm:det-infv}
Let $\Phi_\lambda$ be the determinantal point process with
intensity~$\lambda$ described in the beginning of this section.
Suppose that $\Bar{F_P}$ is regularly varying with index~$-\alpha$ for
$\alpha\in(1,2)$ and let $g$ in \eqref{eq:TildeI_lambda} be an
asymptotic inverse of $1/\Bar{F_P}$.
Then, as $\lambda\to\infty$, $\{\Tilde{I_\lambda}(z)\}_{z\in\R^d}$
converges in finite dimensional distributions to the same
$\alpha$-stable random field as in Theorem~\ref{thm:Poi-infv}.
\end{theorem}

Note that no additional assumption (like \eqref{eq:AssumptionA}) is
required in this case.

\begin{proof}
The proof is almost the same as that of Theorem~\ref{thm:det-fv}.
Only the difference is as follows.
Now, $g(\lambda)$ is regularly varying with index~$1/\alpha$ and can
be represented as $g(\lambda) = \lambda^{1/\alpha}\,L_0(\lambda)$ with
a slowly varying function $L_0$.
Therefore, in~\eqref{eq:det-fv-pr5}, since
$\int_{\R^d}|K_\lambda(0,y)|^2\,\dd y \le K_\lambda(0,0) = \lambda$ by
Remark~\ref{rem:MiyoShir17}, 
\[
  \frac{1}{g(\lambda)^2}
  \int_{\R^d} |K_\lambda(0,y)|^2\,\dd y
  \le \frac{\lambda^{1-2/\alpha}}{L_0(\lambda)^2},
\]
which goes to $0$ as $\lambda\to\infty$ since $\alpha\in(1,2)$.
\end{proof}

\section{Conclusion}\label{sec:conclusion}

In this work, we have considered a shot-noise field defined on a
stationary determinantal point process and have shown that its
centralized and suitably scaled version converges in finite
dimensional distributions to i)~a Gaussian random field when the
amplitudes have the finite second moment and ii)~an $\alpha$-stable
random field when the amplitudes follow a regularly varying
distribution with index~$-\alpha$ for $\alpha\in(1,2)$.
Some extensions can be considered as future work.
For example, as \cite{HeinSchm85} considers a shot-noise field defined
on a Brillinger mixing point process and shows the convergence in
distribution at one position, our result may be extended to the case
of a more general point process.
Furthermore, as \cite{HeinSchm85} and \cite{Inal12} use Berry-Esseen
bound to discuss the rate of convergence, the rates of the
convergences in Theorems~\ref{thm:det-fv} and~\ref{thm:det-infv} may
be interesting challenges.

\bibliographystyle{plain}
\small
%\bibliography{../references}

\end{document}